\documentclass[12pt]{amsart}
\usepackage{color}
\usepackage[utf8]{inputenc}
\usepackage[T1]{fontenc}
\usepackage{amsmath}
\usepackage{amsfonts}
\usepackage{amssymb}
\usepackage{amsthm}
\usepackage[initials]{amsrefs}

\usepackage{mathrsfs}
\usepackage{extarrows}
\usepackage{tikz-cd}
\usepackage[all]{xy}
 \usepackage[english]{babel}
 \usepackage{dsfont}

\makeatletter
\@namedef{subjclassname@2010}{%
  \textup{2010} Mathematics Subject Classification}
\makeatother

\theoremstyle{plain}
\newtheorem{thm}{Theorem}

\theoremstyle{definition}
\newtheorem{defn}[thm]{Definition}
\newtheorem{lem}[thm]{Lemma}

\theoremstyle{remark}
\newtheorem{rem}[thm]{Remark}

\newcommand{\betti}{\beta_{(2)}}
\newcommand{\im}{\operatorname{im}}
\newcommand{\trace}{\operatorname{tr}}
\renewcommand{\dim}{\operatorname{dim}}
\newcommand{\id}{\operatorname{Id}}
\newcommand{\SL}{\operatorname{SL}}

\newcommand{\R}{{\mathbb R}}
\newcommand{\C}{{\mathbb C}}

\newcommand{\N}{{\mathbb N}}

\newcommand{\HX}{HX_{(2)}}
\newcommand{\CX}{CX_{(2)}}
\newcommand{\Lloc}{L^2_{loc}}

\begin{document}

\baselineskip=17pt

\title[$\ell^2$-Betti numbers and coarse equivalence]{Vanishing of $\ell^2$-Betti numbers of locally compact groups as an invariant of coarse equivalence}

\author[R. Sauer]{Roman Sauer}
\address{Institute for Algebra and Geometry \\ Karlsruhe Institute of Technology \\ Englerstr. 2 \\ 76128 Karlsruhe, Germany}
\email{roman.sauer@kit.edu}

\author[M. Schrödl]{Michael Schrödl}
\address{Institute for Algebra and Geometry \\ Karlsruhe Institute of Technology \\ Englerstr. 2 \\ 76128 Karlsruhe, Germany}
\email{michael.schroedl@kit.edu}

\begin{abstract}
We provide a proof that the vanishing of $\ell^2$-Betti numbers of unimodular locally compact second countable groups is an invariant of coarse equivalence. To this end, we define coarse $\ell^2$-cohomology for locally compact groups and show
that it is isomorphic to continuous cohomology for unimodular groups and invariant under
coarse equivalence. 
\end{abstract}



 \maketitle

\section{Introduction}

The insight that the vanishing of $\ell^2$-Betti numbers provides a quasi-isometry invariant is due to Gromov (see~\cite{gromov}*{Chapter~8} for a statement without proof), and positive results around this insight have a long history. The most important contribution is by Pansu~\cite{pansu} whose work on asymptotic $\ell^p$-cohomology includes a proof that the 
vanishing of $\ell^2$-Betti numbers of discrete groups of type $F_\infty$, is 
a quasi-isometry invariant. 

There is a growing interest in the metric geometry of locally compact groups~\cites{cornulier-book, amhyp}. We thus think it is important to have the quasi-isometry and coarse invariance of the vanishing of $\ell^2$-Betti numbers available in the greatest generality. Following Pansu's ideas and relying on more recent advances in the theory of $\ell^2$-Betti numbers, we provide a proof of the following result. 

\begin{thm}\label{thm: main}
Let $G$ and $H$ be unimodular locally compact second countable groups. If $G$ and $H$ are coarsely equivalent then the $n$-th $\ell^2$-Betti number of $G$ vanishes if and only the $n$-th $\ell^2$-Betti number of $H$ vanishes.
\end{thm}

The coarse invariance for discrete groups was proven earlier 
in a paper of Mimura-Ozawa-Sako-Suzuki~\cite{taka}*{Corollary~6.3}. 

Every locally compact, second countable group $G$ (hereafter abbreviated by $\textbf{lcsc}$) has a left-invariant proper continuous metric by a theorem of Struble~\cite{struble}. As any two left-invariant proper continuous metrics on $G$ are coarsely equivalent, every lcsc group has a well defined coarse geometry.     
Further, any coarse equivalence between compactly generated lcsc groups is a quasi-isometry with respect to word metrics of compact symmetric generating sets and vice versa. In particular, a coarse equivalence between finitely generated discrete groups is a quasi-isometry. See~\cite{cornulier-book}*{Chapter~4} for a systematic discussion of these notions. 

To even state Theorem~\ref{thm: main} in that generality, recent advances in the theory of $\ell^2$-Betti numbers were necessary. $\ell^2$-Betti numbers of 
discrete groups enjoy a long history but it was not until recently that $\ell^2$-Betti numbers were defined for arbitrary unimodular lcsc groups by Petersen~\cite{petersen}, and a systematic theory analogous to the discrete case emerged~\citelist{\cite{petersen}\cite{petersen+kyed+vaes}\cite{petersen+sauer+thom}}. 
Earlier studies of $\ell^2$-Betti numbers
of locally compact groups in specific cases can be found
in~\cites{gaboriau-unimodular, dymara, davisetal}.

\subsubsection*{Previous results on coarse invariance}
Pansu~\cite{pansu} introduced asymptotic $\ell^p$-co\-homo\-logy of discrete groups and proved
its invariance under quasi-isometries. If a group $\Gamma$ is of type $F_\infty$, then the $\ell^p$-cohomology of $\Gamma$ coincides with its asymptotic $\ell^p$-cohomology~\cite{pansu}*{Th\'{e}or\`{e}me~1}. The geometric explanation for the appearance of the type $F_\infty$ condition is that the finite-dimensional skeleta of the universal covering of a classifying space of finite type are uniformly contractible. As an immediate consequence of Pansu's result, the vanishing of $\ell^2$-Betti numbers is a quasi-isometry invariant among discrete groups of type $F_\infty$. The same arguments work for 
 totally disconnected groups admitting a topological model of finite type~\cite{sauer-survey}.

Elek~\cite{elek} investigated the relation between $\ell^p$-cohomology of discrete groups
and Roe's coarse cohomology and proved
similar results. Another independent treatment is due to Fan~\cite{fan}. Genton~\cite{genton} elaborated upon Pansu's methods in the case of metric measure spaces. 

Oguni~\cite{oguni} generalised the quasi-isometry invariance of the vanishing of $\ell^2$-Betti numbers from discrete groups of type $F_\infty$ to discrete groups whose cohomology with coefficients in the group von Neumann algebra satisfies a certain technical condition. A similar technical condition appears in the
proof of quasi-isometry invariance of Novikov-Shubin invariants of amenable groups~\cite{sauer-homology}*{}, and it is unclear how much this condition differs from the type $F_\infty$-condition.
Oguni's groupoid approach is inspired by~\cites{gaboriau, sauer-homology} and quite different from the approaches by Elek, Fan, and Pansu. 

The coarse invariance of vanishing of $\ell^2$-Betti numbers for discrete groups was shown by Mimura-Ozawa-Sako-Suzuki~\cite{taka}*{Corollary~6.3}. 
Li~\cite{li} recently reproved this using groupoid techniques as a consequence of more general cohomological coarse invariance results.

\subsubsection*{Structure of the paper} We review the necessary basics of $\ell^2$-Betti numbers and continuous cohomology in Section~\ref{sec: continuous cohomology}. 
In Section~\ref{sec: coarse cohomology} we define coarse $\ell^2$-cohomology for lcsc groups and show that it is isomorphic to continuous cohomology. In Section~\ref{sec: quasiisometry} we conclude the proof of Theorem~\ref{thm: main} and discuss what fails for non-unimodular groups.

\section{Continuous cohomology and $\ell^2$-Betti numbers of lcsc groups}\label{sec: continuous cohomology}

Let $G$ be a unimodular lcsc group with Haar measure $\mu$. 
Let $X$ be a locally compact second countable space with Radon measure $\nu$. Let $E$ be a Fr\'{e}chet space. 

The space $C(X, E)$ of continuous functions from $X$ to $E$ becomes 
a Fr\'{e}chet space when endowed with the topology of compact convergence. 
Let $\Lloc(X, E)$ be the space of 
equivalence classes of measurable maps $f\colon X\to E$ up to $\nu$-null sets such that $||f\vert_K||_E$ 
is square-integrable for every compact subset $K\subset X$. The $L^2$-norm of the function $||f\vert_K||_E$ defines a semi-norm $p_K$ on $\Lloc(X, E)$. The family of semi-norms $p_K$, $K\subset E$, turns 
$\Lloc(X, E)$ into a Fr\'{e}chet space. 

We call a Fr\'{e}chet space $E$ with a continuous (i.e.~$G\to E$, $g\mapsto gv$,	 is continuous for every $v\in E$) linear $G$-action a \emph{$G$-module}.  A continuous linear $G$-equivariant map between $G$-modules is a \emph{homomorphism of $G$-modules}. 
If $E$ is a $G$-module and $G$ acts continuously and $\nu$-preserving on $X$ then $C(X, E)$ and 
$\Lloc(X, E)$ become $G$-modules via 
$(g\cdot f)(x)=gf(g^{-1}x)$ for $x\in X$ and $g\in G$~\cite{blanc}*{Proposition~3.1.1}. 
The usual homogeneous coboundary map
    \begin{equation}\label{eq: coboundary operator}
    d^{n-1} f(g_0,...,g_n)=\sum_{i=0}^n(-1)^i f(g_0,...,\widehat{g_i},...,g_n)
    \end{equation}
defines cochain complexes $C(G^{*+1},E)$ and $\Lloc(G^{\ast +1},E)$ of $G$-modules (cf.~\cite{blanc}*{Proposition~3.2.1}). Here we take the diagonal $G$-action on $G^{\ast+1}$. We recall the following definition. 
\begin{defn}
The \emph{(continuous) cohomology} of $G$ with coefficients in $E$ is the 
cohomology \[ H^n(G, E)=H^n\bigl( C(G^{*+1},E)^G\bigr)\]
of the $G$-invariants of $C(G^{*+1},E)$. The \emph{reduced (continuous) cohomology} $\underline{H}^*(G,E)$ is a quotient of $H^\ast(G,E)$ obtained by taking the quotient with the closure of $\im d^{\ast-1}$ instead of $\im d^{\ast-1}$.
\end{defn}

We have an obvious inclusion 
\begin{equation}\label{eq: comparision map}
 I^*:C(G^{*+1},E)\to L_{loc}^2(G^{*+1},E).
 \end{equation}
The maps $I^\ast$ form a cochain map of $G$-modules. Taking 
a positive function $\chi\in C_c(G)$ there is a cochain map $R^*:L_{loc}^2(G^{*+1},E)\to C(G^{*+1},E)$ of $G$-modules 
\[ (R^nf)(g_0,...,g_n)=\int_{G^{n+1}}f(h_0,...,h_n)\chi(g_0^{-1}h_0)\cdot ...\cdot\chi(g_n^{-1}h_n)d\mu(h_0,...,h_n) \]
such that $I^*\circ R^*$ and $R^*\circ I^*$ are homotopic (as cochain maps of $G$-modules) to the identity \cite[Proposition 4.8]{blanc}. 
So we have the following useful fact: 

\begin{thm}\label{thm: iso from continuous to loc}
The cochain map $I^\ast$ in~\eqref{eq: comparision map} induces 
isomorphisms in cohomology and in reduced cohomology. 
\end{thm}

Next we turn to the case where the coefficient module $E=L^2(G)$ is the 
regular representation, relevant for the definition of $\ell^2$-Betti numbers. 

Let $L(G)$ be the \emph{von Neumann algebra of $G$}; the Haar measure $\mu$ defines a semifinite trace $\trace_\mu$ on $L(G)$. There are a natural left $G$-action and a natural right $L(G)$-action on $L^2(G)$, and the two actions commute. Hence also the $G$-actions on 
$C(G^{\ast+1}, L^2(G))$ and $\Lloc(G^{\ast+1}, L^2(G))$ considered previously and the $L(G)$-actions induced from the right $L(G)$-action on $L^2(G)$ commute. So the (reduced and non-reduced) continuous cohomology of $G$ with coefficients in $L^2(G)$ is naturally a $L(G)$-module\footnote{When talking about $L(G)$-modules we mean the algebraic module structure and ignore topologies.}. Obviously, the cochain map $I^\ast$ above is compatible with the $L(G)$-module structures. The groups 
	$H^\ast(G, L^2(G))$ are called the \emph{(continuous)} $\ell^2$-cohomology of $G$. Similarly for the reduced cohomology. 

  Petersen~\cite{petersen} extended L\"uck's dimension function from finite von Neumann algebras to semifinite von Neumann algebras. The dimension function $\dim_\mu$ with respect to $(G,\mu)$ is a non-trivial dimension for (algebraic) right $L(G)$-modules that is additive for short exact sequences of $L(G)$-modules.
    It scales as $\dim_{c\mu}=c^{-1}\dim_{\mu}$ for $c>0$.
    The fact that a $L(G)$-module has dimension zero can be expressed without referring to the trace: it is an algebraic fact.
    The following criterion was shown by the first author for finite von Neumann algebras~\cite{sauer-groupoids}*{Theorem~2.4}; it was extended to the semifinite case by Petersen~\cite{petersen}*{Lemma~B.27}.

    \begin{thm}\label{thm: local criterion}
	    An $L(G)$-module $M$ satisfies $\dim_\mu(M)=0$ if and only if
	    for every $x\in M$ there is an increasing sequence $(p_i)$ of projections in $L(G)$ with $\sup p_i=1$
	    such that $xp_i=0$ for every $i\in \N$.
    \end{thm}
\begin{defn}
    The $n$-th $\ell^2$-\emph{Betti number of} $G$ is the
    $L(G)$-dimension of its reduced continuous cohomology with coefficients in $L^2(G)$, i.e.
    \[ \betti^n(G):=\dim_\mu \underline{H}^n(G,L^2(G))\in [0,\infty].\]
\end{defn}
\begin{rem}\label{rem: equivalent definitions of l2-Betti}
Equivalently, the $n$-th $\ell^2$-Betti number can be defined
as the $L(G)$-dimension of the
non-reduced cohomology $H^n(G, L^2(G))$. This is a non-trivial fact  (see~\cite{petersen+kyed+vaes}*{Theorem~A}). For discrete $G$, our definition coincides with L\"uck's definition in~\cite{luck}. Again, this is non-trivial and shown in~\cite{peterson+thom}*{Theorem~2.2}.
\end{rem}

The following lemma was observed in~\cite{petersen}*{Proposition~3.8}.
Since it is a direct consequence of Theorem~\ref{thm: local criterion} we present the argument.

\begin{lem}\label{lem: vanishing of reduced cohomology and l2 betti}
$\betti^n(G)=0 \Leftrightarrow \underline{H}^n(G, L^2(G))=0.$
\end{lem}

\begin{proof}
Let $\betti^n(G)=0$. Let $f\colon G^{n+1}\to L^2(G)$ be a cocycle representing a cohomology class $[f]$ in $\underline{H}^n(G, L^2(G))$.
By Theorem~\ref{thm: local criterion} there is an increasing sequence of projections $p_j\in L(G)$ whose supremum is $1$ such that each $fp_j$ is a coboundary $d^{n-1}b_j$. It is clear that $fp_j=d^{n-1}b_j$ converges to $f$ in the topology of $C(G^{n+1}, L^2(G))$, thus $[f]=0$.
\end{proof}

\section{Coarse equivalence and coarse $\ell^2$-cohomology}\label{sec: coarse cohomology}

Let $G$ be a lcsc group. We fix a left-invariant proper continuous metric $d$ on $G$. Let $\mu$ be a Haar measure on $G$. Let $\mu_{n}$ be the $n$-fold product measure of $\mu$ on $G^n$. 

 For every $R>0$ and $n\in\mathbb{N}_0$ we consider the closed subset
\[G_R^n:=\lbrace(g_0,...,g_{n-1})\in G^{n}\;|\;d(g_i,g_j)\leq R\text{ for all }0\leq i,j\leq n\!-\!1 \rbrace\]
and a family of semi-norms for measurable maps $\alpha\colon G^{n+1}\to\mathbb{C}$ defined by
\[\|\alpha\|_R^2=\int_{G_R^{n+1}}|\alpha(g_0,...,g_n)|^2d\mu_{n+1}\in [0,\infty].\]

Let $\CX^{n}(G)$ be the space of equivalence 
classes (up to $\mu_{n+1}$-null sets) of measurable maps $\alpha\colon G^{n+1}\to\mathbb{C}$  such that $\|\alpha\|_R<\infty$ for every $R>0$. The semi-norms $\|\_\|_R$, $R>0$, turn $\CX^{n}(G)$ into a 
Fr\'{e}chet space. It is straightforward to verify that the 
homogeneous differential~\eqref{eq: coboundary operator} 
yields a well-defined, continuous homomorphism 
$\CX^{n}(G)\to \CX^{n+1}(G)$ (cf.~\cite{genton}*{Proposition~2.3.3}). 
Thus we obtain a cochain complex of Fr\'{e}chet spaces. 

\begin{defn} The \emph{coarse $\ell^2$-cohomology of $G$} is defined as 
\[\HX^n(G)=H^n\bigl(\CX^\ast(G)\bigr).\]
By taking the quotients by the closure of the differentials, one defines similarly the \emph{reduced coarse $\ell^2$-cohomology} $\underline{H}X_{(2)}^n(G)$. 
\end{defn}

\begin{rem}
The previous definition is the continuous analog of Elek's definition~\cite{elek}*{Definition~1.3} in the discrete case (Elek gives credits to Roe~\cite{roe}). 
It is very much related to Pansu's \emph{asymptotic $\ell^2$-cohomology}~\cite{pansu}, which was considered in the generality of metric measure spaces 
by Genton~\cite{genton}. The difference of our definition to the one in Genton~\cite{genton} is as follows: $\CX^\ast(G)$ is an inverse limit of spaces $L^2(G^{\ast+1}_R)$. Unlike us, Genton takes first the cohomology of $L^2(G^{\ast+1}_R)$ and then the inverse limit. Under some uniform contractibility assumptions the two definitions coincide but likely not in general. 
\end{rem}

 \begin{thm}\label{thm:iso of continuous and coarse coho}
 Let $G$ be a unimodular lcsc group. For every $n\ge 0$, 
 the $n$-th continuous cohomology of $G$ with coefficients in the left regular representation $L^2(G)$ is isomorphic to the $n$-th coarse $\ell^2$-cohomology of~$G$, and likewise for reduced cohomology. 
  \end{thm}

\begin{proof}
We have the obvious embedding 
\[\Lloc(G^{n+1}, L^2(G))\subset \Lloc(G^{n+1}, \Lloc(G))\] 
and the exponential law (see~\cite{blanc}*{Lemme 1.4} for a proof but beware of the typo in the statement)
\[ \Lloc(G^{n+1}, \Lloc(G))\cong \Lloc(G^{n+1}\times G).\]
Thus an element in $\Lloc(G^{n+1}, L^2(G))^G$ is represented by a 
measurable complex function in $(n+2)$-variables. 
For $\alpha\in \Lloc(G^{n+1}, L^2(G))^G$ we define $\mu_{n+2}$-almost everywhere 
\[F^{n}(\alpha)(x_0,\ldots, x_n, x)=\alpha(x^{-1}x_0,\ldots, x^{-1}x_n)(x).\]
The measurable function $F^{n}(\alpha)$ is invariant by translation in the $(n+2)$-th variable. 
By Fubini's theorem we may regard $F^{n}(\alpha)$ as 
a measurable function $E^{n}(\alpha)\colon G^{n+1}\to\C$ in the first $(n+1)$-variables. We may think of $E^n(\alpha)$ as an evaluation of $\alpha$ at $e$. 
Let $B(R)$ denote the $R$-ball around $e\in G$. 
Next we show 
that $\|E^n(\alpha)\|_R<\infty$ for every $R>0$, thus $E^n(\alpha)\in \CX^n(G)$. 

Since $\alpha\in \Lloc(G^{n+1}, L^2(G))^G$ we have 
\begin{align*}
 \infty&>	 \int_{B(2R)^{n+1}} \int_{G}
		 |\alpha(x_0,x_1,...,x_n)(x)|^2d\mu d\mu_{n+1}\\
	   &= \int_{B(2R)^{n+1}} \int_{G}
		 |\alpha(x,xx_0^{-1}x_1,...,xx_0^{-1}x_n)(x_0)|^2d\mu d\mu_{n+1}.
\end{align*}
The map 
\[m\colon G^{n+2}\to G^{n+2}, (x_0,\ldots, x_n, x)\mapsto (x, xx_0^{-1}x_1, \ldots, xx_0^{-1}x_n, x_0)\]
is measure preserving since it is the composition of taking inverses in the last coordinate, left multiplication by $xx_0^{-1}$, conjugation by $x$ and taking inverses in the last coordinate. 
Note that this requires unimodularity. 
Further, we have 
\[ m\bigl( G_R^{n+1}\times B(R)\bigr)\subset B(2R)^{n+1}\times G.\] 
This implies the first inequality below. The first equality follows from the fact that $(x_0,\dots, x_n,x)\mapsto (x^{-1}x_0, \dots, x^{-1}x_n,x)$ is a measure preserving measurable automorphism of $G_R^{n+1}\times B(R)$. 
\begin{align*}
	   \int_{B(2R)^{n+1}} \int_{G}
		 |\alpha(x,xx_0^{-1}x_1&,...,xx_0^{-1}x_n)(x_0)|^2d\mu d\mu_{n+1}\\
	   &\ge \int_{G_R^{n+1}}\int_{B(R)} |\alpha(x_0, \dots, x_n)(x)|^2d\mu d\mu_{n+1} \\
	   &=\int_{G_R^{n+1}}\int_{B(R)} |\alpha(x^{-1}x_0, \dots, x^{-1}x_n)(x)|^2d\mu d\mu_{n+1} \\
	   &=\mu(B(R))\|E^n(\alpha)\|_R.
\end{align*}

Hence $\|E^n(\alpha)\|_R$ is finite for every $R>0$. 
That 
\[E^\ast\colon \Lloc(G^{\ast+1}, L^2(G))^G\to \CX^\ast(G)\] 
defines a cochain map is obvious. The above computation also implies that 
$E^\ast$ is continuous with respect to the Fr\'{e}chet topologies. 

Given $\beta\in \CX^n(G)$ we define 
\[M^n(\beta)(g_0,\ldots, g_n)(g)=\beta(g^{-1}g_0, \ldots, g^{-1}g_n)\]
for $\mu_{n+2}$-almost every $(g_0,\ldots, g_n,g)$. The function 
$M^n(\beta)$ defines an element in $\Lloc(G^{n+1}, L^2(G))^G$. The $G$-invariance of $M^n(\beta)$ is obvious. We have to show that $\|M^n(\beta)\vert_{B(R)^{n+1}}\|$ is square-integrable for every $R>0$. This  follows from the following computations which is based on the arguments above in reversed order. 

\begin{align*}
\mu(B(R))\int_{G_{2R}^{n+1}} |&\beta(g_0,...,g_n)|^2 d\mu_{n+1}\\
	      & =\int_{G_{2R}^{n+1}}\int_{B(R)} |\beta(g_0,...,g_n)|^2 d\mu d\mu_{n+1}\\
	      & \ge\int_{B(R)^{n+1}}\int_G |\beta(g^{-1}g_0,\dots, g^{-1}g_n)|d\mu  d\mu_{n+1}
\end{align*}
Obviously, $M^\ast$ is a chain map. Continuity follows from the previous computation. It is clear that $M^\ast$ and $E^\ast$ are mutual inverses. Using 
Theorem~\ref{thm: iso from continuous to loc}, this concludes the proof. 
\end{proof}

\section{Coarse invariance}\label{sec: quasiisometry}
We recall the notion of coarse equivalence. A map $f\colon (X,d_X)\to (Y,d_Y)$ between metric spaces is \emph{coarse Lipschitz} if there 
is a non-decreasing function $a\colon [0,\infty)\to [0,\infty)$ with $\lim_{t\to\infty} a(t)=\infty$ such that 
\[ d_Y(f(x), f(x'))\le a(d(x,x'))\] 
for all $x,x'\in X$. We say that two such maps $f,g$ are \emph{close} 
if \[\sup_{x\in X} d_Y(f(x), g(x))<\infty.\]  
A coarse Lipschitz map $f\colon X\to Y$ is a \emph{coarse equivalence} if there is a coarse Lipschitz map $g\colon Y\to X$ such that $fg$ and $gf$ are close to the identity. We say $g$ is a \emph{coarse inverse} of $f$. 
\begin{lem}\label{coarse equivalences are measurable}
Coarsely equivalent lcsc groups are measurably coarse equivalent, i.e.~if $G$ and $H$ are coarse equivalent lcsc groups then there are \emph{measurable} coarse Lipschitz maps $f\colon G\to H$ and $g\colon H\to G$ such that $fg$ and $gf$ are close to the identity.
\end{lem}

\begin{proof}
We choose left-invariant continuous proper metrics $d_G$ and $d_H$ on $G$ and $H$, respectively. Let $f\colon G\to H$ be a coarse Lipschitz map with $d_H(f(x), f(x'))\leq a(d_G(x,x'))$. Let $t>0$. We pick a countable measurable partition $\mathcal{U}$ of $G$ whose elements have diameter $\le t$ and choose an element $x_U\in U$ for every $U\in \mathcal{U}$.

By setting $\tilde{f}(x)=f(x_U)$ for $x\in U$ we obtain a coarse Lipschitz map $\tilde{f}\colon G\to H$ which satisfies 
$d(\tilde{f}(x),\tilde{f}(x'))\le a(d(x,x')+2t)$
and is close to $f$ with $d(\tilde{f}(x),f(x))\leq a(2t)$. Analogously, we construct  a measurable coarse Lipschitz map $\tilde{g}$, constructed from a coarse Lipschitz map $g\colon H\to G$ which is a coarse inverse to $f$. It is obvious that $\tilde{g}$ is a coarse inverse to $\tilde{f}$.
\end{proof}
\begin{thm}\label{qi invariance of HX}
Coarsely equivalent lcsc groups have isomorphic reduced and non-reduced coarse $\ell^2$-cohomology groups.
\end{thm}
\begin{proof}
Let $G$ and $H$ lcsc groups with Haar measures $\mu$ and $\nu$, respectively. Let  $f\colon G\to H$ be a coarse equivalence with coarse inverse $g$. Because of lemma \ref{coarse equivalences are measurable} we can further assume that $f$ and $g$ are measurable. We define a map $\chi\colon G\times G \to \mathbb{R}$ by
\begin{align*}
\chi(x,y)=\frac{\mathds{1}_{B_x(c)}(y)}{\mu(B(c))}
\end{align*}
where we choose $c$ such that $\mu(B(c))\ge 1$. Then $\chi$ is a measurable function with
$\chi(x,y)=\chi(y,x)$ and $\int_G \chi(x,y)d\mu(y)=1$. We use the following notation:
\[\chi\colon G^{n+1}\times G^{n+1}\to \mathbb{R},\;\;\;\chi((x_0,...,x_n),(y_0,...,y_n))=\chi(x_0,y_0)\cdot...\cdot \chi(x_n,y_n).\]
 Analogously, we define $\chi'\colon H^{n+1}\times H^{n+1} \to \mathbb{R}$ with some radius $c'$. Now we can define the maps $f^*\colon \HX^*(H)\to \HX^*(G)$ and $g^*\colon\HX^*(G)\to \HX^*(H)$ as follows where we use $x_i$ for elements in $G$ and $y_i$ for elements of $H$:
\begin{align*}
f^*\alpha(x_0,...,x_n)&=\int_{H^{n+1}} \alpha(y_0,...,y_n)\chi'\left((f(x_0),...,f(x_n)),(y_0,...,y_n)\right)d\nu_{n+1}\\
g^*\beta(y_0,...,y_n)&=\int_{G^{n+1}} \beta(x_0,...,x_n)\chi\left((g(y_0),...,g(y_n)),(x_0,...,x_n)\right)d\mu_{n+1}.
\end{align*}
The idea of averaging over a function like $\chi$ goes back to Pansu; it is necessary in our context since the maps $f$ and $g$ do not preserve the measure classes, in general. 

First of all, we check that these are well-defined continuous cochain maps.
\begin{align*}
&\infty >\|\alpha\|^2_{a(R)+c'}=\int_{H^{n+1}}|\alpha(y_0,...,y_n)|^2\cdot \mathds{1}_{H^n_{a(R)+c'}}d\nu_{n+1} \\
&\ge \int_{H^{n+1}}|\alpha(y_0,...,y_n)|^2\int_{G_R^{n+1}}\chi'\left((f(x_0),...,f(x_n)),(y_0,...,y_n)\right)d\mu_{n+1}d\nu_{n+1}\\
&=\int_{G_R^{n+1}}\int_{H^{n+1}}\left| \alpha(y_0,...,y_n)\right|^2\chi'\left((f(x_0),...,f(x_n)),(y_0,...,y_n)\right)d\nu_{n+1}d\mu_{n+1} \\
&\ge \int_{G_R^{n+1}}\left|\int_{H^{n+1}} \alpha(y_0,...,y_n)\chi'\left((f(x_0),...,f(x_n)),(y_0,...,y_n)\right)d\nu_{n+1}\right|^2d\mu_{n+1} \\
&=\int_{G_R^{n+1}}|f^n\alpha(x_0,...,x_n)|^2d\mu_{n+1} = \|f^n\alpha\|^2_{R}
\end{align*}
It is a direct computation that $d^n\circ f^n=f^{n+1}\circ d^n$.

It remains to show that there is a cochain homotopy $h\colon\CX^*(H)\to \CX^{*-1}(H)$ such that $\id -g^*f^*=hd+dh$. We define $h_i^{n+1}\colon\CX^{n+1}(H)\to \CX^{n}(H)$ by
\begin{align*}
&h_i^{n+1}\alpha(y_0,...,y_n)\\
=&\int_{H^{n+1}}\alpha(\tilde{y}_0,...,\tilde{y}_i,y_i,...,y_n)\chi'((y_0,...,y_n),(\tilde{y}_0,...,\tilde{y}_n))d\nu_{n+1}(\tilde{y})
\end{align*}
and set 
\[h^{n+1}=\sum_{i=0}^n(-1)^ih_i^{n+1}.\]
That $h^*$ is well-defined is a similar consideration as to show that $f^*$ and $g^*$ are well-defined. Now let us denote the i-th term of the coboundary map by $d^n_i$, i.e. $d^n_i\alpha(y_0,...,y_{n+1})=\alpha(y_0,...,\widehat{y_i},...,y_{n+1})$. It is straightforward to verify that we have the following relations:
\begin{align*}
h_{n}^{n+1}\circ d^{n}_{n+1}&= g^n\circ f^n, \\
h_0^{n+1}\circ d^n_{0} 		&= \id_{\CX^n(H)}, \\
h_j^{n+1}\circ d_i^n &= d^{n-1}_i\circ h_{j-1}^n &\text{    for } 1\leq j\leq n \text{ and } i\leq j,& \\
h_j^{n+1} \circ d_i^n & = d_{i-1}^{n-1}\circ h_j^n &\text{    for } 1\leq i \leq n \text{ and } i>j.&
\end{align*}
We get $h^{n+1}d^n+d^{n-1}h^n=\id_{\CX^n(H)} - g^n\circ f^n$.  The same construction applies to $f^*g^*$ which completes the proof.
\end{proof}
\begin{proof}[Proof of Theorem~\ref{thm: main}]
Let $G$ and $H$ be unimodular lcsc groups. Let $G$ and $H$ be coarsely equivalent. Then we have the following equivalences:
\begin{align*}
\beta_{(2)}^2(G)=0&\Leftrightarrow\underline{H}^n(G,L^2(G))=0 &(\text{Lemma }\ref{lem: vanishing of reduced cohomology and l2 betti})\\
&\Leftrightarrow \underline{H}X_{(2)}^n(G)=0 &(\text{Theorem }\ref{thm:iso of continuous and coarse coho}) \\
&\Leftrightarrow \underline{H}X_{(2)}^n(H)=0 &(\text{Theorem }\ref{qi invariance of HX})
\end{align*}
Going the same steps backwards for the group $H$ finishes the proof. 
\end{proof}

\begin{rem}
Since the Borel subgroup $B<\SL_2(\R)$ of upper triangular matrices is cocompact, the solvable Lie groups $B$ 
and $\SL_2(\R)$ are quasi-isometric. So $B$ belongs to the class of amenable hyperbolic lcsc groups of which a systematic study was undertaken in~\cite{amhyp}.

The group $B$ is not unimodular and thus its $\ell^2$-Betti number are not defined. Nevertheless, one may ask what exactly breaks down in the proof above which can be formulated to a large part without the notion of $\ell^2$-Betti numbers. By a result of Delorme~\cite{delorme}*{Corollaire V.3}, we have $\underline{H}^1(B, L^2(B)))=0$. Since Theorem~\ref{qi invariance of HX} does not require unimodularity, we have $\underline{H}X_{(2)}^1(B)\cong \underline{H}X_{(2)}^n(\SL_2(\R))\ne 0$ since $\betti^1(\SL_2(\R))\ne 0$. So it is Theorem~\ref{thm:iso of continuous and coarse coho} that fails for the non-unimodular group $B$. 
\end{rem}

\subsection*{Acknowledgements}
We acknowledge support by the German Science Foundation via the Research Training Group 2229.

\end{document}